\documentclass[graybox]{svmult}

\usepackage{mathptmx} % selects Times Roman as basic font
\usepackage{helvet} % selects Helvetica as sans-serif font
\usepackage{courier} % selects Courier as typewriter font
\usepackage{type1cm} % activate if the above 3 fonts are
\usepackage{multicol} % used for the two-column index
\usepackage[bottom]{footmisc}% places footnotes at page bottom

\usepackage{amssymb}
\usepackage{mathrsfs}%, amssymb, array}%euscript}
\usepackage[all,cmtip,matrix, arrow]{xy}

\newcommand{\operatorname}[1]{\mathrm{#1}}

\renewcommand\labelenumii{\textup{(\alph{enumii})}}

\newtheorem{ccorollary}[theorem]{Corollary}
\newtheorem{llemma}[theorem]{Lemma}
\newtheorem{pproposition}[theorem]{Proposition}

\begin{document}
\title*{On standard models of conic fibrations over a field of characteristic zero}

\author{Artem Avilov}

\institute{
National Research University Higher School of Economics,
\email{v07ulias@gmail.com}
}

\maketitle
\abstract{
In 1982 V.G. Sarkisov proved the existense of standard models of conic fibrations over algebraically closed fields of $\operatorname{char}\neq 2$. In this paper we will prove the analogous result for three-dimensional conic fibrations over arbitrary fields of characteristic zero with a finite group action.
}

\markright{Standard models of conic fibrations}
\footnotetext[0]{The author was partially supported by RFFI grants 11-01-00336 and 12-01-31012, and by grants MK-6223.2012.1 and MK-1192.2012.1}
\section{Introduction}\label{s1}
This paper is motivated by birational classification of threefolds over non-algebrai\-cally closed fields and threefolds with a group action. The minimal model program gives us the first step towards such a classification -- using some elementary birational transformations, every three-dimensional variety can be reduced to a variety whose canonical class is nef or to a variety with a structure of a Fano fibration with some additional properties (cf. ~\cite[\S 2.2]{5} or ~\cite[\S 3]{6}). In this paper we consider varieties with a structure of a fibration by rational curves. Algebraic surfaces with a structure of a fibration by rational curves are studied very well (cf. ~\cite{3}).

In 1982 V.G. Sarkisov in his paper ~\cite[Theorem 1.13]{1} proved that every conic fibration over an algebraically closed field of $\operatorname{char}\neq 2$ has a standard model, that is a Mori fibration by rational curves with smooth total space which is birationally equivavalent to the original one. The anologous theorem for three-dimensional conic fibrations over algebraically closed fields of $\operatorname{char}\neq 2$ was proved by A.A. Zagorskii ~\cite[Theorem 1]{2} in 1977, but his proof contains some gaps. The aim of this paper is the analogous result for three-dimensional varieties with a structure of a fibration by rational curves over arbitrary fields of characteristic zero with a finite group action. The main result of this paper is the following theorem: 
\begin{theorem}\label{th1} Let $k$ be an arbitrary field of characteristic 0. Let $X$ be a geometrically irreducible three-dimensional algebraic variety over $k$, let $Y$ be an algebraic surface over $k$, and let $f: X \dasharrow Y$ be a dominant rational map such that the generic fiber is a rational curve over $k(Y)$. Suppose that a finite group $G$ acts on $X$ and $Y$ by birational automorphisms such that the map $f$ is equivariant. Then the triple  $(X, Y, f)$ has a \textnormal{standard model}, i. e. there exists a commutative diagram 
$$
\xymatrix{
X \ar@{-->}[r]\ar@{-->}[d]^{f} & X' \ar[d]^{f'}\\
Y \ar@{-->}[r]& Y'
}
$$
where $X'$ and $Y'$ are smooth projective varieties with an action of $G$, maps $X\dasharrow X'$ and $Y\dasharrow Y'$ are birational, $f'$ is a Mori fibration and all maps are $G$-equivariant.
\end{theorem}

We hope that this paper will be useful for the birational classification of threefolds over non-algebraically closed fields and for classification of finite subgroups in the Cremona group $\operatorname{Cr}_{3}(k)$ (cf.~\cite{9}, ~\cite{10}). We expect that the analogous result holds in arbitrary dimension, but methods of this paper don't work in higher dimensions.

Author is grateful to his scientific adviser Yu. Prokhorov for posing the problem and for numerous useful discussions and advices, to V. Shokurov and C.Shramov for useful discussions and to K. Khabrov for his help with english version of this paper.
\section{Preliminaries}\label{s2}\setcounter{theorem}{0}
Throughtout the paper all objects are defined over an arbitrary field $k$ of characteristic zero. By a \emph{geometrically irreducible} variety over $k$ we mean a variety $X$ such that $X\otimes\bar{k}$ is irreducible over $\bar{k}$. We say that a curve $C$ over an arbitrary field $K$ is \emph{rational} if $C\otimes_{K} \bar{K}$ is birationally equivalent to $\mathbb{P}^{1}_{\bar{K}}$.
\subsection{$G$-varieties and regular conic fibrations}\label{s2.1}
In this paper we use the language of $G$-varieties (cf. ~\cite{3}). Here we will give only the most necessary definitions and statements.

Let $G$ be a finite group.
\refstepcounter{theorem}
\begin{definition}\label{de1} A \emph{$G$-variety (or a variety with an action of the group $G$)} is a pair $(X, \rho)$ where $X$ is an algebraic variety and $\rho:G \to \operatorname{Aut}_{k}(X)$ is a group homomorphism.
\end{definition}
\refstepcounter{theorem}
\begin{definition}\label{de2} A triple $(X, Y, f)$ is a \emph{fibration by rational curves} over the base $Y$ if $X$ is a geometrically irreducible threefold, $Y$ is a geometrically irreducible surface and $f:X \dasharrow Y$ is a rational map such that its general fiber is a rational curve over $k(Y)$. We say that a fibration $(X, Y, f)$ is a \emph{conic fibration} if the surface $Y$ is smooth, the map $f$ is a projective morphism and any scheme fiber of $f$ over an arbitrary closed point of $Y$ is isomorphic to a conic in $\mathbb{P}^2$ (not neccessary smooth).
\end{definition}
\refstepcounter{theorem}
\begin{definition}\label{de3} The group $G$ \emph{acts on a fibration} $(X, Y, f)$ if $X$ and $Y$ are equipped with structures of $G$-varieties and $f$ is a $G$-equivariant map. We call such fibrations $G$-\emph{fibrations}.
\end{definition}
\refstepcounter{theorem}
\begin{definition}\label{de4} A conic fibration $(X, Y, f)$ is \emph{regular} if the map $f$ is a flat morphism between smooth varieties.
\end{definition}
\begin{pproposition}\label{pr1}\textnormal{(cf. ~{\cite[Prop. 1.2]{4}})} Let $f:X\to Y$ be a regular conic fibration. Then the following statements hold:\begin{enumerate}
\item
The sheaf $\mathcal{E}=f_{*}\mathcal{O}(-K_{X})$ is locally free of rank 3;
\item
$\mathcal{O}(-K_{X})$ is a relatively very ample sheaf, and it defines an embedding of the variety $X$ into $\mathbb{P}(\mathcal{E})=\operatorname{Proj}(S^{\bullet}\mathcal{E})$ such that every fiber of $f$ maps to a conic in the corresponding projective plane;
\item
There exists a discriminant curve (this curve may be reducible) $\Delta_{f}$ with the following properties:
\begin{enumerate}
\renewcommand\labelenumii{(\roman{enumii})}
\item
$\Delta_{f}$ has only simple normal crossings of its irreducible components as singularities;
\item
Let $y$ be a closed point of a surface $Y$ and let $X_{y}$ be the corresponding fiber of $f$. Then $y \notin \Delta_{f}$ iff $X_{y}$ is a smooth conic, $y \in \Delta_{f} \setminus \operatorname{Sing} \Delta_{f}$ iff $X_{y}$ is a geometrically reducible reduced conic and $y\in \operatorname{Sing} \Delta_{f}$ iff $X_{y}$ is a double line.
\end{enumerate}
\end{enumerate}
\end{pproposition}
\refstepcounter{theorem}
\begin{remark}\label{re1} In the paper ~\cite{4} this statement was proved for the case of an algebraically closed field, but the case of a non-algebraically closed field can be reduced easily to the previous case.
\end{remark}
\refstepcounter{theorem}
\begin{remark}\label{re2} The image of the embedding $X\subset \mathbb{P}(\mathcal{E})$ from Proposition ~\ref{pr1} is an embedded conic fibration in the sense of Definition ~\ref{de7} and in this case the discriminant curve coincides with the discriminant divisor (see Definition ~\ref{de8}).
\end{remark}
\refstepcounter{theorem}
\begin{definition}\label{de5} A regular conic fibration $(X, Y, f)$ with an action of the group $G$ is a \emph{standard conic fibration} if varieties $X$ and $Y$ are projective and the morphism $f$ is relatively minimal, i.e. for every $G$-invariant $G$-irreducible divisor $D \subset Y$  ($G$-irreducible means that it cannot be represented as a sum of two non-zero $G$-invariant effective divisors) its preimage $f^{-1}(D)$ is a $G$-irreducible divisor on $X$. This property is equivalent to $\rho(X/Y)^G=1$, where $\rho(X/Y)^G$ is the rank of the $G$-invariant part of the relative Picard group.
\end{definition}
\refstepcounter{theorem}
\begin{definition}\label{de6} $G$-fibrations by rational curves $(X, Y, f)$ and $(X', Y', f')$ are \emph{equivalent} if there exist birational  $G$-equivariant maps $\lambda :X\dasharrow X'$ and $\mu :Y \dasharrow Y'$ such that the following diagram is commutative:
$$
\xymatrix{
X \ar@{-->}[r]^{\lambda}\ar@{-->}[d]^{f} & X' \ar@{-->}[d]^{f'}\\
Y \ar@{-->}[r]^{\mu}& Y'
}
$$
\end{definition}
\begin{theorem}\label{th2}\textnormal{(cf. ~{\cite{1}})} Every fibration by rational curves $(X, Y, f)$ over an algebraically closed field of $\operatorname{char}\neq2$ has a standard model $(X', Y', f')$, i.e. a standard conic fibration which is equivalent to the original one.
\end{theorem}
\subsection{Embedded conic fibrations}\label{s2.2}
\refstepcounter{theorem}
\begin{definition}\label{de7} Let $\mathcal{E}$ be a locally free sheaf of rank 3 on a smooth surface $Y$ and let $\tau : \mathbb{P}(\mathcal{E}) \to Y$ be a standard projection. Then we say that an irreducible reduced divisor $X \subset \mathbb{P}(\mathcal{E})$ is an \emph{embedded conic fibration} if the generic fiber of the induced map $\tau|_{X}:X\to Y$ is a conic over $k(Y)$ (some fibers over closed points can be two-dimensional).
\end{definition}
\refstepcounter{theorem}
\begin{remark} An embedded conic fibration isn't necessary a conic fibration in the sense of Definition ~\ref{de2}.
\end{remark}

Let $\mathcal{E}$ be a locally free sheaf of rank 3 on $Y$, let $\mathcal{L}=\mathcal{O}_{\mathbb{P}(\mathcal{E})}(1)$ and let $\mathcal{M}$ be an invertible sheaf on $Y$. Suppose that zeroes of some section $\sigma \in H^{0}(\mathbb{P}(\mathcal{E}), \mathcal{L}^{2}\otimes \tau^{*}\mathcal{M})$ form an embedded conic fibration. Obviously, every embedded conic fibration can be obtained in the same way for suitable $\mathcal{E}$ and $\mathcal{M}$. Moreover, if $(X, Y, f)$ is a regular conic fibration and $X \subset \mathbb{P}(\mathcal{E})$ is a relatively anticanonical embedding (cf. Proposition ~\ref{pr1}) then we can put $\mathcal{E}=f_{*}\mathcal{O}_{X}(-K_{X})$ and then $\mathcal{M}$ has the canonical isomorphism with $\det(\mathcal{E}^{*})\otimes \mathcal{O}_{Y}(-K_{Y})$ (for details see ~\cite{1}).

There exists the natural isomorphism $$H^{0}(Y, S^{2}(\mathcal{E})\otimes \mathcal{M})\cong H^{0}(\mathbb{P}(\mathcal{E}), \mathcal{L}^{2}\otimes \tau^{*}(\mathcal{M})),$$ so we can denote the corresponding section of the sheaf $S^{2}(\mathcal{E})\otimes \mathcal{M}$ by the same letter $\sigma$. There exists another natural monomorphism $S^{2}(\mathcal{E})\otimes \mathcal{M} \to \operatorname{Hom}(\mathcal{E}^{*}, \mathcal{E} \otimes \mathcal{M})$, so $\sigma$ defines a morphism of sheafs $q(\sigma): \mathcal{E}^{*}\to \mathcal{E}\otimes\mathcal{M}$. The morhism $q(\sigma)$ defines a morhism of vector bundles $q_{0}(\sigma):\Lambda^{3}\mathcal{E}^{*}\to \Lambda^{3}(\mathcal{E}\otimes\mathcal{M})$.
\refstepcounter{theorem}
\begin{definition}\label{de8} The divisor of zeroes of the morhism of vector bundles $$q_{0}(\sigma)\in H^{0}(Y, (\det\mathcal{E})^{2}\otimes\mathcal{M}^{3})\subset H^{0}(Y, \operatorname{Hom}(\det\mathcal{E}^{*}, \det(\mathcal{E}\otimes\mathcal{M})))$$ is a \emph{discriminant divisor} of the conic fibration.
\end{definition}
\refstepcounter{theorem}
\begin{remark}\label{re3} The discriminant divisor isn't reduced in general case, but it is reduced for regular conic fibrations (cf. ~\cite[corollary 1.9]{1}).
\end{remark}
\refstepcounter{theorem}
\begin{definition}\label{re4} Let $f: X\to Y$ be a conic fibration and $U\subset Y$ be an open subset. Assume that $\operatorname{codim}(Y\setminus U, Y)>1$ and the induced map $f_{U}: X_{U}\to U$ defines a regular conic fibration. Let $\Delta\subset U$ be a discriminant divisor of $f_{U}$. Then we call its Zariski closure $\bar{\Delta}\subset Y$ a \emph{discriminant divisor} of the conic fibration $(X, Y, f)$.
\end{definition}
\subsection{Minimal model program}\label{s2.3}
\refstepcounter{theorem}
\begin{definition} Let $X$ be a $G$-variety. We say that $X$ has \emph{$G\mathbb{Q}$-factorial singularities} if every $G$-invariant Weyl divisor is  $\mathbb{Q}$-Cartier.
\end{definition}\label{de9}
\refstepcounter{theorem}
\begin{definition} We say that a morphism between algebraic varieties is a  \emph{contraction} if it has connected fibers.
\end{definition}
\refstepcounter{theorem}
\begin{definition} A projective $G$-equivariant contraction $\pi: X\to Y$ is called an \emph{extremal $G$-equivariant Mori contraction} if $X$ is a normal variety with at worst terminal $G\mathbb{Q}$-factorial singularities, the anticanonical class is relatively ample and all $G$-invariant exceptional curves are numerically proportional.
\end{definition}
\refstepcounter{theorem}
\begin{definition} An extremal $G$-equivariant Mori contraction $\pi: X\to Y$ is called a \emph{$G$-Mori fibration} if $\operatorname{dim} X>\operatorname{dim} Y$, a \emph{divisorial} contraction if $\operatorname{dim} X>\operatorname{dim} Y$ and the exceptional set is a divisor and a \emph{small} contraction if $\operatorname{dim} X>\operatorname{dim} Y$ and the exceptional set is a curve.
\end{definition}

\begin{theorem}\label{th3}\textnormal{ (cf. ~\cite[\S 2.2]{5})}
Let $X$ be a three-dimensional $G$-variety with at worst $G\mathbb{Q}$-factorial terminal singularities, let $Y$ be a $G$-surface and let $f:X\to Y$ be a projective $G$-equivariant morphism. Then after a sequence of equivariant divisorial contractions and flips we obtain a projective $G$-equivariant morphism $f':X'\to Y$ such that $X'$ has at worst $G\mathbb{Q}$-factorial terminal singularities and either the canonical divisor $K_{X'}$ is relatively nef (so we say that $X'$ is a relatively minimal model over $Y$) or the morphism $f'$ factors through a $G$-Mori fibration $g:X'\to Z$.
\end{theorem}
\refstepcounter{theorem}
\begin{definition}\label{de11} A linear system of Weyl divisors is \emph{movable} if it has no base components. Let $\mathcal{L}_{i}$ be a movable linear systems on $X$ and let $\mathcal{L}=\sum l_{i}\mathcal{L}_{i}$ be their formal linear combination. We say that a birational morphism $Y\to X$ is a \emph{log-resolution of singularities} of the pair $(X, \mathcal{L})$ if $Y$ is smooth and preimages of all linear systems $\mathcal{L}_{i}$ has no base points. Let $D_{i}\in\mathcal{L}_{i}$ be general members of linear systems. We say that the pair $(X, \mathcal{L})$ has \emph{terminal} (resp. \emph{canonical}) singularities if for every log-resolution of singularities $f: Y \to X$ in the formula 
$$K_{Y}+D_{Y}=f^{*}(K_{X}+D)+\sum \lambda_{j}E_{j},$$ one has all $\lambda_{j}>0$ (resp. all $\lambda_{j}\geq 0$), where $D=\sum l_{i}D_{i}$ and $D_{Y}$ is a strict transform of $D$.
\end{definition}

The following useful category was introduced by V.A. Alexeev in ~\cite{13}:
\refstepcounter{theorem}
\begin{definition}\label{de12} The category $\mathbb{Q}LSc$ consists of pairs $(X, \mathcal{L})$ where $X$ is a $\mathbb{Q}$-factorial variety with at worst terminal singularities and $ \mathcal{L}=\sum l_{i} \mathcal{L}_{i}$ is a formal sum of movable linear systems with non-negative coefficients $l_{i}$ such that the pair $(X, \mathcal{L})$ is canonical.
\end{definition}

By ~\cite{13} all fundamental theorems and hypotheses of the minimal model program hold in the category $\mathbb{Q}LSc$  in the three-dimensional case (the cone theorem, the contraction theorem, the existence of flips and the termination of a sequence of flips). So the minimal model program works well in this category. If $X$ is a three-dimensional $G$-variety and every linear system $\mathcal{L}_{i}$ is $G$-invariant then the $G$-equivariant relative minimal model program (cf. Theorem ~\ref{th3}) also works in the category $\mathbb{Q}LSc$.

The next theorem is a $G$-equivariant version of a partial crepant resolution of singularities of a pair (cf. ~\cite[Proposition 2.10]{17}):
\begin{theorem}\label{th4} Suppose that a $G$-variety $X$ has at worst $G\mathbb{Q}$-factorial terminal singularities and $\mathcal{H}$ is a $G$-invariant movable linear system with base points. Then there exists a $G$-variety $Z$ with at worst $G\mathbb{Q}$-factorial terminal singularities, an extremal $G$-equivariant divisorial Mori contraction $p:Z\to X$ and a number $c$ such that the pair $(Z, cp^{-1}_{*}\mathcal{H})$ has at worst canonical singularities and $$K_{Z}+cp^{-1}_{*}\mathcal{H}\sim p^{*}(K_{X}+c\mathcal{H}).$$
\end{theorem}
\begin{proof}
There exists a number $c$ (the canonical threshold of the pair) such that the pair $(X, c\mathcal{H})$ is canonical but not terminal. Let $(Y, c\mathcal{H}_{Y})$ be an equivariant resolution of singularities for the pair $(X, c\mathcal{H})$. Now we apply the equivariant relative minimal model program to the pair $(Y, (c+\epsilon)\mathcal{H}_{Y})$ with respect to the morphism $Y \to X$ in the category $\mathbb{Q}LSc$ with a suitable $\epsilon$ and obtain a pair $(Y', c\mathcal{H}_{Y}')$ with a morphism $Y'\to X$ such that the exceptional set of this morphism consists exactly of crepant divisors of the pair $(X, c\mathcal{H})$ and $Y'$ has at worst terminal singularities. Applying further equivariant relative minimal model program for the variety $Y'$ we obtain the required pair $(Z, cp^{-1}_{*}\mathcal{H})$ on the last step (for details see ~\cite{13} and ~\cite[Proposition 2.10]{17}). \hfill$\Box$
\end{proof}

Also we need the following theorem of Yu.G. Prokhorov and S. Mori on the structure of three-dimensional Mori fibrations with a two-dimensional base:
\begin{theorem}\label{th5}\textnormal{[cf. ~\cite[Theorem 1.2.7]{8}]} Let $f:X\to Z$ be a three-dimensional Mori fibration with a two-dimensional base over an algebraically closed field of characteristic zero. Then $Z$ has at worst Du Val singularities of type A.
\end{theorem}
\section{The proof of the main theorem}\label{s3}\setcounter{theorem}{0}

\begin{theorem}\label{th6} Let $k$ be an arbitrary field of characteristic 0. Let $X$ be a three-dimensional algebraic geometrically irreducible variety over $k$, let $Y$ be a surface over $k$, let $G$ be a finite group acting by birational automorphisms on $X$ and $Y$ and let $f: X \dasharrow Y$ be a $G$-equivariant dominant map which defines a fibration by rational curves. Then the $G$-fibration $(X, Y, f)$ has a standard model, i.e. there exists a standard $G$-conic fibration which is equivalent to the original one.
\end{theorem}
\begin{proof} Let us break the proof into several lemmas. 
\begin{llemma}\label{le1} Suppose that a fibration $(X, Y, f)$ satisfies the conditions of Theorem ~\ref{th6}. Then there exists a $G$-Mori fibration by rational curves $(X',Y', f')$ which is equivalent to $(X, Y, f)$, where $X'$ and $Y'$ are projective. Moreover, $Y'$ has at worst Du Val singularities.
\end{llemma}
\begin{proof} Replacing $X$ and $Y$ by their open subsets we may assume that $X$ and $Y$ are normal and that $G$ acts on $X$ and $Y$ by regular automorphisms. Quotient varieties $X/G$ and $Y/G$ are quasiprojective, let $\bar{X}$ and $\bar{Y}$ be their projective closures. Then a normalisation of $\bar{X}$ (resp. $\bar{Y}$) in the function field $k(X)$ (resp, $k(Y)$) is a projective variety with a regular action of $G$ which is birationally equivalent to $X$ (resp. $Y$). So we can assume from the begining that $X$ and $Y$ are projective and $G$ acts on them by regular automorphisms. 

Applying the equivariant resolution of singularities of the varieties $Y$ , $X$ and of the map $f$ (it is possible due to ~\cite{12}) we can assume that $f$ is a morphism between nonsingular projective varieties and the generic fiber of $f$ is a rational curve. Applying the equivariant relative (over $Y$) minimal model program we obtain a $G\mathbb{Q}$-factorial three-dimensional variety $X'$ over $Y$ with at worst terminal singularities and one of two alternative possibilities happens: either the canonical divisor $K_{X'}$ is $f'$-ample with respect to the morphism $f': X'\to Y$ or the morphism $f'$ factors through a $G$-Mori fibration $f'':X'\to Y'$. But initially the divisor $-K_{f^{-1}(U)}$ is $f$-ample for an open subset $U$ such that the fibration over $U$ is a regular conic fibration, so the first possibility cannot happen. Thus we have that $(X', Y', f')$ is a $G$-Mori fibration by rational curves because of dimensional reasons. Also we have that $X'$ has at worst terminal singularities and $Y'$ is a surface with at worst Du Val singularities by Theorem ~\ref{th5}. \hfill$\Box$
\end{proof}
\begin{llemma}\label{le2} Let $(X', Y', f')$ be a $G$-Mori fibration over the two-dimensional base. Then it is equivalent to a $G$-fibration by rational curves $(\widetilde{X}, \widetilde{Y}, \widetilde{f})$, where $\widetilde{X}$ is a projective threefold, $\widetilde{Y}$ is a smooth surface, $\widetilde{f}$ is a morphism  and there exists a divisor $\widetilde{\Delta}$ on $\widetilde{Y}$ with simple normal crossings such that  $\widetilde{f}$ is smooth over $\widetilde{Y}\setminus\widetilde{\Delta}$.
\end{llemma}
\begin{proof}
Throwing out a finite number of points from $Y'$, more presicely all singular points of $Y'$ and images of all singular points of $X'$ (singularities of $X'$ are terminal and therefore isolated) we obtain an open subset $U\subset Y'$ such that our fibration is a regular conic fibration over $U$. So we can define the discriminant divisor on $U$. Let us denote by $\Delta'$ its closure in $Y'$, this divisor is reduced according to Remark ~\ref{re3}. Let us consider an equivariant resolution of singularities of the pair $(Y', \Delta')$, let us denote it by $\widetilde{\alpha}:\widetilde{Y}\to Y'$ . Now the divisor $\widetilde{\Delta}=\Delta'+\operatorname{Exc}(\widetilde\alpha)$ is a reduced divisor with simple normal crossings. Let $\widetilde{X}$ be an equivariant resolution of singularities of the dominant component of $X'\times_{Y'} \widetilde{Y}$ (at first we normalize this component if it's necessary) and  $\widetilde{f}$ be the corresponding morphism from $\widetilde{X}$ to $\widetilde{Y}$. Obviously, all degenerate fibers are laying over points of $\widetilde{\Delta}$. \hfill$\Box$
\end{proof}
\refstepcounter{theorem}
\begin{definition} By a \emph{$G$-point} of a $G$-variety we mean a $G$-orbit of some closed point.
\end{definition}
\begin{llemma}\label{le3} Let $(W, V, g)$ be a $G$-Mori fibration with the two-dimensional base. Let $v$ be a singular $G$-point of the surface $V$. Then there exists another $G$-Mori fibration $(W', V', g')$ which is eqivalent to $(W, V, g)$, such that the map $V'\to V$ is a partial crepant resolution of singularity at the $G$-point $v$ and the map $W\dasharrow W'$ is a Sarkisov link (for the definition see ~\cite[\S 13]{6}). 
\end{llemma}
\begin{proof} Let us consider a linear system of hyperplane sections $\mathcal{H}$ for some embedding of $V$ into a projective space such that $\operatorname{Bs}(\mathcal{H})$ coinsides with $v$ and $\mathcal{H}$ is invariant under the action of the group $G$. Such linear system can be obtained in the following way: let us consider a very ample divisor $D$ on $V$ which is invariant under the action of the group $G$ and let $\mathcal{H}$ be a finite-dimensional $G$-invariant linear subsystem of divisors containing the $G$-point $v$ in the full linear system $|nD|$ for $n$ big enough. Let $\mathcal{N}=g^{-1}_{*}(\mathcal{H})$ be the preimage of $\mathcal{H}$ and let $c$ be the canonical threshold of the pair $(W, \mathcal{N})$. According to Theorem ~\ref{th4} the pair $(W, c\mathcal{N})$ has a partial crepant resolution of singularities $(\widetilde{W}, c\widetilde{\mathcal{N}})$. Denote the morphism $\widetilde{W}\to W$ by $\pi$.

Consider the equivariant relative (over $V$) minimal model program for the pair $(\widetilde{W}, c\widetilde{\mathcal{N}})$. All curves on the exceptional divisor $E$ of the morphism $\pi$ are numerically proportional and have a zero intersection number with $K_{\widetilde{W}}+ c\widetilde{\mathcal{N}}$ by the projection formula. Let $\widetilde{C}$ be a sufficienly general fiber of the map $h:\widetilde{W}\to V$ and let $C$ be a fiber of the map $W\to V$ over the same point, then by the projection formula we have $$\widetilde{C}\cdot (K_{\widetilde{W}}+c\widetilde{\mathcal{N}})=\widetilde{C}\cdot\pi^{*}(K_{W}+c\mathcal{N})=C\cdot(K_{W}+c\mathcal{N})=C\cdot K_{W}<0,$$  where the last equality follows from the choice of the linear system. So we have exactly one negative extremal ray on the relative $G$-invariant Mori cone $\operatorname{NS}(\widetilde{W}/V)^{G}$ and its contraction gives us either a $G$-Mori fibration by rational curves or a small contraction (in fact, the first case is impossible but now it's not important for us). 

In the second case after a sequence of $G$-equivariant log-flips (this sequence cannot be infinite, see ~\cite[\S 6.3]{5} or ~\cite[\S 9.2]{6}) we get a $G$-variety $\widetilde{W}'$ with a $G$-equivariant morphism $h':\widetilde{W}'\to V$. The relative $G$-invariant Mori cone $\operatorname{NS}(\widetilde{W}'/V)^{G}$ is generated by two extremal rays and exactly one of them is negative. The contraction of the negative ray can be either a divisorial contraction or a $G$-Mori fibration. Suppose that we have the first case, so we obtain a $G$-variety $\widehat{W}$ with a morphism $\widehat{g}:\widehat{W}\to V$ and this morphism gives us a $G$-Mori fibration over the surface $V$. But then $W$ and $\widehat{W}$ are $G$-Mori fibrations over the common base $V$ which are isomorphic in codimension 1, thus actually they are isomorphic: $$W\simeq \operatorname{Proj}(\bigoplus \mathcal{O}_{W}(-nK_{W}))\simeq \operatorname{Proj}(\bigoplus \mathcal{O}_{\widehat{W}}(-nK_{\widehat{W}})) \simeq \widehat{W}.$$ Notice that the divisor $E$ is the only divisor in fibers of $h':\widetilde{W}'\to V$, so only $E$ can be the exceptional divisor of the contraction $\widetilde{W}'\to\widehat{W}$. But then on one hand $E$ is crepant for the pair $(W, c\mathcal{N})$ by the construction and on the other hand $E$ is an exceptional divisor on the contraction of the negative ray, so it cannot be crepant. Contradiction. 

So the case of a divisorial contraction is impossible, and we obtain a $G$-Mori fibration $\widetilde{g}:\widetilde{W}\to \widehat{V}$ with the morphism $\sigma: \widehat{V}\to V$. The morphism $\widetilde{g}$ must be a fibration by rational curves by the dimensional reason. Thus we have the following commutative $G$-equivariant diagram which defines a Sarkisov link:
$$
\xymatrix{
(W, c\mathcal{N})\ar[d]^{g} & (\widetilde{W}, c\widetilde{\mathcal{N}}) \ar[l] \ar@{-->}[r] & (\widehat{W}, c\widehat{\mathcal{N}})\ar[d]^{\widehat{g}}\\
V & & \widehat{V}\ar[ll]
}
$$
Moreover, $\widehat{V}$ has at worst Du Val singularities by Theorem ~\ref{th5} and the exceptional set of the morphism $\widehat{V} \to V$ consists of the unique $G$-invariant $G$-irreducible divisor (because the exceptional divisor of $\pi: \widetilde{W} \to W$ cannot lie in a fiber of the morphism $\widehat{g}$). By~\cite[Theorem 1.4]{7} this morphism can be either a composition of weighted blow-ups of smooth points of $V$ or a crepant partial resolution of singularities of $V$. Obviously, we have the second case because the exceptional divisor lie over the singular $G$-point $v$ by the construction, so the surface $\widehat{V}$ is a crepant partial resolution of singularities of the surface $V$. \hfill$\Box$
\end{proof}
\begin{ccorollary}\label{co1} Let $(W, V, g)$ be a $G$-Mori fibration with the two-dimensional base. Then there exists an equivalent $G$-Mori fibration $(W', V', g')$ with the two-dimensio\-nal base such that the surface $V'$ is the minimal resolution of singularities of the surface $V$.
\end{ccorollary}
\begin{proof}
Applying Lemma ~\ref{le3} several times we get a $G$-Mori fibration over a smooth two-dimensional base, because the number of crepant divisors on $V$ is finite (all crepant divisors are exactly exceptional divisors on the mininal resolution of singularities of $V$). \hfill$\Box$
\end{proof}
\begin{llemma}\label{le4} Suppose that fibration $f: X\dasharrow Y$ satisfies the conditions of Theorem ~\ref{th6}. Then there exists a $G$-Mori fibration $\widehat{f}:\widehat{X}\to\widehat{Y}$ such that this fibration is equivalent to the original one, $\widehat{Y}$ is a smooth projective surface and the discriminant divisor is a reduced divisor with only simple normal crossings of its components as singularities.
\end{llemma}
\begin{proof} Applying Lemma ~\ref{le1} we get a projective $G$-Mori fibration $f':X'\to Y'$ which is equivalent to the original one. Then we apply Lemma ~\ref{le2} to the $G$-fibration $f':X'\to Y'$ and get a $G$-fibration by rational curves $\widetilde{f}: \widetilde{X}\to\widetilde{Y}$ such that $\widetilde{Y}$ is a projective smooth surface and there exists a divisor $\widetilde{\Delta}$ with simple normal crossings such that the morphism $\widetilde{f}$ is smooth over $\widetilde{Y}\setminus\widetilde{\Delta}$. Now we apply the relative equivariant minimal model program to $\widetilde{X}\to \widetilde{Y}$. For the same reasons as in Lemma ~\ref{le1} we get a $G$-Mori fibration by rational curves $\bar{f}: \bar{X}\to \bar{Y}$ with a morphism $\bar{\alpha}:\bar{Y}\to \widetilde{Y}$. $\bar{X}$ has at worst terminal singularities and $\bar{Y}$ has at worst Du Val singularities. 

Now we apply Corollary ~\ref{co1} to the $G$-fibration $\bar{f}: \bar{X}\to \bar{Y}$. We obtain a $G$-Mori fibration $\widehat{f}:\widehat{X}\to\widehat{Y}$ with a morphism $\widehat{\alpha}:\widehat{Y}\to\bar{Y}$. So we have the following $G$-equivariant commutative diagram:
$$
\xymatrix{
X\ar[d]^{f} \ar@{-->}[r] & X' \ar[d]^{f'} & \widetilde{X}\ar[d]^{\widetilde{f}}\ar[l]\ar@{-->}[r] & \bar{X} \ar[d]^{\bar{f}}\ar@{-->}[r] & \widehat{X} \ar[d]^{\widehat{f}}\\
Y & Y'\ar@{-->}[l]& \widetilde{Y} \ar[l]^{\widetilde{\alpha}}& \bar{Y} \ar[l]^{\bar{\alpha}}& \widehat{Y} \ar[l]^{\widehat{\alpha}}
}
$$

Now we consider the divisor $\widehat{\Delta}=\widetilde{\Delta}+\operatorname{Exc}(\bar{\alpha}\circ\widehat{\alpha})$. Then $\widehat{\Delta}$ is a divisor with simple normal crossings because the divisor $\widetilde{\Delta}$ has the same property and $\bar{\alpha}\circ\widehat{\alpha}$ is a morphism between smooth surfaces. The discriminant divisor of the fibration $(\widehat{X}, \widehat{Y}, \widehat{f})$ (let us denote it by $\widehat{\Delta}'$) is a reduced divisor and all of its components are contained in $\widehat{\Delta}$. So, $\widehat{\Delta}'$ is a reduced divisor with only simple normal crossings of its components as singularities. \hfill$\Box$
\end{proof}
\refstepcounter{theorem}
\begin{definition}\label{de13} A $G$-\emph{sheaf} is a quasicoherent sheaf $E$ on a $G$-variety $Z$ with the set of an isomorphisms of sheaves $\lambda_{g}^{E}:E\to g^{*}E$ satisfying the following relations: $\lambda_{1}^{E}=Id_{E}$ and $\lambda_{g_{1}g_{2}}^{E}=g^{*}_{2}(\lambda_{g_{1}}^{E})\circ \lambda_{g_{2}}^{E}$.
\end{definition}
\begin{llemma}\label{le5} Let $g:W\to U$ be a regular $G$-conic fibration where $U$ is an open subset of a nonsingular surface $V$ such that the complement $V\setminus U$ consists of a finite number of points. Then there exists an embedded $G$-conic fibration $g':W'\to V$ isomorphic to $g:W\to U$ over $U$.
\end{llemma}
\begin{proof}
Denote the natural embedding $U\to V$ by $j$. By Proposition ~\ref{pr1} the anticanonical divisor $-K_{W_{U}}$ induces an embedding of $W$ into $\mathbb{P}(\mathcal{E}_{1})$, where $\mathcal{E}_{1}=g_{*}\mathcal{O}_{W}(-K_{W})$.

Now we notice that the dualizing sheaf $\omega_{W}$ has a natural structure of a $G$-sheaf on $W$ and it induces a structure of a $G$-sheaf on the anticanonical sheaf $-K_{W}$. Consequently, the locally free sheaf $\mathcal{E}$ has a natural structure of a $G$-sheaf. Let us consider the sheaf $\mathcal{E}=(j_{*}\mathcal{E}_{1})^{\vee\vee}$. It is a locally free sheaf of rank 3 on $V$ (because it is a reflexive sheaf and every reflexive sheaf on a smooth surface is actually locally free, cf.~\cite[Corollary 1.4]{11}), and it also has a structure of a $G$-sheaf. Moreover, it is the only locally free sheaf which extends $\mathcal{E}_{1}$ on the whole $V$. Let us denote the closure of $W$ in $\mathbb{P}(\mathcal{E})$ by $W'$. It is an embedded conic fibration and the action of the group $G$ on $\mathbb{P}(\mathcal{E})$ induces the action of $G$ on $W'$. \hfill$\Box$
\end{proof}
\begin{llemma}\label{le6} Let $h: W\to V$ be a three-dimensional $G$-Mori fibration over the smooth two-dimensional base such that the discriminant divisor $\Delta$ is a reduced divisor with simple normal crossings. Then there exists a locally free $G$-sheaf $\mathcal{E}$ of rank 3 on $V$ and an equivariant embedding $W \to \mathbb{P}(\mathcal{E})$ over $V$ which induces a structure of an embedded $G$-conic fibration on $W$.
\end{llemma}
\begin{proof}
If $W$ is smooth then there's nothing to prove by Proposition ~\ref{pr1}. If $W$ is not smooth then all of its singularities are terminal. Terminal three-dimensional singularities are isolated, so we can apply Lemma ~\ref{le5} (where $U$ is a complection to the set of images of all singular points of $W$). Then we apply Lemma ~\ref{le5}, let us denote the corresponding embedded $G$-conic fibration by $W'\subset \mathbb{P}(\mathcal{E})\to V$. Since $V\setminus U$ is a finite set of points then the discriminant divisor of $G$-conic fibration $(W', V, h')$ coincides with $\Delta$. 

Now we verify that the morphism $h'$ is flat. Over a neighborhood $B$ of an arbitrary point $v \in V$ we can define our fibration $W'$ with a quadratic form $$Q_{B}(x_{1}, x_{2}, x_{3})=\sum\limits_{i, j=1}^{3}A_{i, j}x_{i}x_{j}$$ on $\mathbb{P}^{2}\times B$ where $A_{i,j}$ are functions from $k[B]$. If the rank of such form in $v$ is zero then all $A_{i,j}$ vanishes at $v$ and $\operatorname{det}Q_{B}$ lies in $\mathfrak{m}_{v}^{3}$, where $\mathfrak{m}_{v}$ is the maximal ideal of the point $v$. But $\operatorname{det}Q_{B}$ defines the discriminant divisor which has at worst a simple normal crossing at the point $v$. So, the rank of the form $Q_{B}$ cannot be zero at any point. Variety $W'$ is a normal locally complete intersection and the morphism $h'$ is equidimensional, so actually this morphism is flat. As a consequence, varieties $W$ and $W'$ are isomorphic over $V$ in codimension 1. Anticanonical divisors $-K_{W}$ and $-K_{W'}$ are ample so $$W\simeq \operatorname{Proj}(\bigoplus \mathcal{O}_{W}(-nK_{W}))\simeq \operatorname{Proj}(\bigoplus \mathcal{O}_{W'}(-nK_{W'})) \simeq W'.$$ Thus $W$ is an embedded $G$-conic fibration such that its discriminant divisor is reduced and has only simple normal crossings as singularities. \hfill$\Box$
\end{proof}

So, applying Lemma ~\ref{le4} we obtain a $G$-Mori fibration $\widehat{f}:\widehat{X}\to\widehat{Y}$ with the smooth projective base $\widehat{Y}$ such that the discriminant divisor $\widehat{\Delta}$ is a reduced divisor with simple normal crossings. By Lemma ~\ref{le6} this fibration is an embedded (into a projectivization of some locally free sheaf of rank 3) $G$-conic fibration. If $\widehat{X}$ is smooth then the morphism $\widehat{f}$ is flat by ~\cite[Theorem 3.5]{14}). Assume that $\widehat{X}$ has singular points. In two following lemmas we will construct a Sarkisov link which will help us resolve singularities of $\widehat{X}$.
\begin{llemma}\label{le7} Let $h: W\to V$ be an embedded (into a projectivization of some locally free sheaf of rank 3) conic fibration ($V$ is not necessary projective) such that the discriminant divisor $\Gamma$ is reduced and has only simple normal crossings. Then all singular points of $W$ are ordinary double points. Moreover, if $\pi: \widetilde{W}\to W$ is the blow up of the maximal ideals of all singular points, then $\widetilde{W}$ is smooth and the anticanonical divisor $-K_{\widetilde{W}}$ is relatively nef and big.
\end{llemma}
\begin{proof}
Due to Sarkisov (cf. ~\cite[Corollary 1.11]{1}) all singular points of $W$ lay exactly on geometrically reducible fibers of the fibration over singular points of $\Gamma$. Let $w$ be a singular point of $W$ and let $v$ be a corresponding point on $V$, $v=h(w)$. We claim that in an appropriate formal neighborhood of the point $v$ we can define our fibration $W\to V$ by the following quadratic form 
\begin{equation}
\label{1}
Q(x_{1}, x_{2}, x_{3})=ax_{1}^{2}+bx_{2}^{2}+stx_{3}^{2}=0,
\end{equation}
where $s$ and $t$ are functions which locally define components of $\Gamma$ passing through $v$, and $a$ and $b$ are some nonzero elements of the base field $k$. 

Let us consider a neighborhood $U$ of the point $v$ such that the sheaf from the statement of the lemma is trivial over $U$ and the intersection $U\cap \Gamma$ consists of two components of $\Gamma$ passing through $v$. Then we can define $W$ over this neighborhood by a quadratic form $$Q_{U}(x_{1}, x_{2}, x_{3})=\sum\limits_{i, j=1}^{3}A_{i, j}x_{i}x_{j}$$ in $\mathbb{P}^{2}\times U$ where $A_{i,j}$ are functions from $k[U]$. This quadratic form defines a geometrically reducible conic over $v$, so we can apply a linear change of variables $x_{i}$ (with coefficients in $k$) and obtain the following (in)equalities: $A_{1, 1}(v)=a\neq 0, A_{2, 2}=b\neq 0$ and all other $A_{i,j}$ vanish at the point $v$. Now we apply the Lagrangian method to put our quadratic form into its canonical form (if necessary we can reduce $U$), so $W$ is a subvariety of $\mathbb{P}^{2}\times U$ defined by a quadratic form $$Q_{U}'(x_{1}, x_{2}, x_{3})=Ax_{1}^{2}+Bx_{2}^{2}+Cx_{3}^{2},$$ where $A, B, C$ are elements of $k[U]$ such that $A(v)=a\neq0, B(V)=b\neq0$. But we know that $\operatorname{det}Q_{U}'$ defines the divisor $\Gamma$ in $U$, thus $C=stC'$ where $C'(v)=1$. Now we notice that in a formal neighborhood we can take a square root of every power series such that its free term is 1. Now the first statement of the lemma is an easy consequence of the formula (\ref{1}) which defines $W$ locally.

The last statement is local on the base. Let us consider a neighborhood of the point $v$ which doesnot contain other singular points of $\Gamma$. Then the Mori cone $\operatorname{NS}(\widetilde{W}/V)$ is two-dimensional and we have two extremal rays $R_{1}$ and $R_{2}$ of the cone. The ray $R_{1}$ is generated by a curve on the exceptional divisor $E$ of the morphism $\pi$ and has a negative intersection number with $K_{\widetilde{W}}$ while the second ray $R_{2}$ is generated by the strict transform $\pi^{-1}(l)$ of the fiber of the morphism $W\to V$ over the point $v$. The ray $R_{2}$ is $K_{\widetilde{W}}$-trivial: $$\pi^{-1}_{*}(l)\cdot K_{\widetilde{W}}=\pi^{-1}_{*}(l)\cdot (\pi^{*}(K_{W})+E)=l\cdot K_{W}+\pi^{-1}_{*}(l)\cdot E=-2+2=0.$$ The last equality is true because the curve $l$ is geometrically reducible and both of its components over the algebraic closure of $k$ have a transversal intersection with $E$ in the single point. Thus we have the numerical effectiveness of the anticanonical divisor $-K_{\widetilde{W}}$. The divisor$-K_{\widetilde{W}}$ is ample on the generic fiber, thus it is big. \hfill$\Box$
\end{proof}
\begin{llemma}\label{le8} Let $h: W\to V$ be an embedded conic fibration ($V$ is not necessary projective) such that its discriminant divisor $\Gamma$ is reduced and consists of two smooth irreducible components with transversal intersection at the single point $v$, and let $w$ be a singular point of $W$. Then there exists the following commutative diagram defining a Sarkisov link:
$$
\xymatrix{
W\ar[d]^{h} & \widetilde{W} \ar[l] \ar@{-->}[r] & \widehat{W}\ar[d]^{\widehat{h}}\\
V & & \widehat{V}\ar[ll]
}
$$
where the morphism $\widetilde{W}\to W$ is the blow-up of $w$, the morphism $\widehat{V}\to V$ is the blow-up of $v$, the map $\widetilde{W}\dasharrow \widehat{W}$ is an isomorphism in codimension 1 and $\widehat{W}\to \widehat{V}$ is a flat morphism between smooth varieties.
\end{llemma}
\begin{proof}
Let $v\in H$ be a hyperplane section of $V$. Let  us denote by $H_{n}$ the sum $nh^{*}(H)-K_{W}$ where $n\gg 0$, this divisor is ample for sufficiently big $n$, because $-K_{W}$ is relatively ample and $H$ is ample on the base. There exists a number $l$ such that $lH_{n}$ is very ample. Let us denote by $\mathcal{H}$ a big enough finite-dimensional linear subsystem of the full linear system $|lH_{n}|$ such that every element of $\mathcal{H}$ contains $w$, and let $c$ be a canonical threshold of the pair $(W, \mathcal{H})$. Singularities of the pair $(W, \mathcal{H})$ can be resolved by the single blow-up $p: \widetilde{W}\to W$ by Lemma ~\ref{le7} and by the choice of the linear system. Let $E$ be the exceptional divisor of the blow-up.

Let us consider the relative Mori cone $\operatorname{NS}(\widetilde{W}/V)$. As in Lemma ~\ref{le7} it is generated by two rays, the ray $R_{1}$ is generated by a curve on the exceptional divisor and has a zero intersection number with $p^{*}(K_{W}+c\mathcal{H})$ by the definition of the canonical threshold, the ray $R_{2}$ is generated by the strict transform of the fiber $h^{-1}(v)$, $R_{2}$ is $p^{*}(K_{W}+c\mathcal{H})$-negative. Applying the minimal model program in the category $\mathbb{Q}LSc$ to the pair $(\widetilde{W}, c\mathcal{H})$ we get a log-flip $\widetilde{W}\dasharrow \widehat{W}$ on the first step, let $\widehat{\mathcal{H}}$ be the image of the linear system $p^{-1}_{*}\mathcal{H}$. By Lemma ~\ref{le7} this log-flip is actually a flop, so $\widehat{W}$ is smooth (cf. ~\cite[Theorem 2.4]{16}). 

On the second step of the minimal model program we get a $K_{\widehat{W}}$-negative Mori contraction $\widehat{W}\to \widehat{V}$ (the Mori cone $\operatorname{NS}(\widehat{W}/V)$ has a unique $K_{\widehat{W}}+\widehat{\mathcal{H}}$-negative ray, and actually this ray is also $K_{\widehat{W}}$-negative). This contraction cannot be divisorial by the same reason as in the proof of Lemma ~\ref{le3}. Also this contraction can not be small because for smooth threefolds there are no $K$-negative small extremal contractions (cf. ~\cite[\S 2.1]{16}). So we get a non-singular Mori fibration over two-dimensional base which is also non-singular, cf. ~\cite[Theorem 3.5]{14}. The corresponding morphism between smooth surfaces $\widehat{V}\to V$ has the single exceptional divisor which maps to the point $v$. So we obtain the required diagram. Such diagram is unique because two Mori fibrations over the same base which are isomorphic in codimension 1 are actually isomorphic since they both are isomorphic to $\operatorname{Proj}(\bigoplus \mathcal{O}(-nK_{\widehat{W}}))$. \hfill$\Box$
\end{proof}

Applying Lemma ~\ref{le8} in a neighbourhood of every singular point of $\widehat{X}$ we get a regular conic fibration $\widehat{X}'\to \widehat{Y}$, which has a structure of $G$-conic fibration because the construction in Lemma ~\ref{le8} is canonical. Also it is easy to see, that the rank of the relative invariant Picard group is equal to 1, so  $\widehat{X}'\to \widehat{Y}$ is a standard $G$-conic fibration which is equivalent to the original $G$-fibration by rational curves $X\to Y$. \hfill$\Box$
\end{proof}

% % 
% \bibliography{my_ref,specific/inv_groups}
% \bibliographystyle{alpha}
% % %\bibliographystyle{gost71u}%{alpha}

\def\cprime{$'$} \def\polhk#1{\setbox0=\hbox{#1}{\ooalign{\hidewidth
 \lower1.5ex\hbox{`}\hidewidth\crcr\unhbox0}}}

\end{document}